\documentclass[a4paper]{amsart}
\usepackage{amssymb}
\usepackage{xypic}
\usepackage[only,mapsfrom]{stmaryrd}
\usepackage{graphicx}

\newtheorem{theorem}{Theorem}[section]
\newtheorem{lemma}[theorem]{Lemma}

\newtheorem{prop}[theorem]{Proposition}

\theoremstyle{definition}

\theoremstyle{remark}
\newtheorem{remark}[theorem]{Remark}

\numberwithin{equation}{section}
\newcommand{\Q}{\mathbb{Q}}

\DeclareMathOperator{\codim}{codim}
\DeclareMathOperator{\rank}{rank}

\DeclareMathOperator{\edim}{edim}

\title{Quasitoric manifolds homeomorphic to homogeneous spaces}

\author{Michael Wiemeler}
\address{School of Mathematics,
    Alan Turing Building, The University of Manchester, Oxford Road, Manchester M13 9PL, UK}
\email{michael.wiemeler-2@manchester.ac.uk}
\thanks{Part of the research was supported by SNF Grants Nos. 200021-117701 and 200020-126795.}

\subjclass[2000]{Primary 57S15, 57S25.}

\keywords{quasitoric manifolds, homogeneous spaces, cohomogeneity one manifolds}

\begin{document}
\begin{abstract}
 We present some classification results for quasitoric manifolds \(M\) with \(p_1(M)=-\sum a_i^2\) for some \(a_i\in H^2(M)\) which admit an action of a compact connected Lie-group \(G\) such that \(\dim M/G \leq 1\). In contrast to Kuroki's work \cite{kuroki_pre_1_2009,kuroki_pre_2_2009} we do not require that the action of \(G\) extends the torus action on \(M\).
\end{abstract}

\maketitle

\section{Introduction}

Quasitoric manifolds are certain \(2n\)-dimensional manifolds on which an \(n\)-dimen\-sional torus acts such that the orbit space of this action may be identified with a simple convex polytope.
They were first introduced by Davis and Januszkiewicz \cite{davis91:_convex_coxet} in 1991.

In \cite{kuroki_pre_1_2009,kuroki_pre_2_2009} Kuroki studied quasitoric manifolds \(M\) which admit an extension of the torus action to an action of some compact connected Lie-group \(G\) such that \(\dim M/G \leq1\).
Here we drop the condition that the \(G\)-action extends the torus action in the case where the first Pontrjagin-class of \(M\) is equal to the negative of a sum of squares of elements of \(H^2(M)\).
In this note all cohomology groups are taken with coefficients in \(\Q\).
We have the following two results.

\begin{theorem}
\label{thm1}
  Let \(M\) be a quasitoric manifold with \(p_1(M)=-\sum a_i^2\) for some \(a_i\in H^2(M)\)  which is homeomorphic (or diffeomorphic) to a homogeneous space \(G/H\) with \(G\) a compact connected Lie-group.
  Then \(M\) is homeomorphic (diffeomorphic) to \(\prod S^2\).
  In particular, all Pontrjagin-classes of \(M\) vanish.
\end{theorem}

\begin{theorem}
\label{sec:thm2}
  Let \(M\) be a quasitoric manifold with \(p_1(M)=-\sum a_i^2\) for some \(a_i\in H^2(M)\).
  Assume that the compact connected Lie-group \(G\) acts smoothly and almost effectively on \(M\) such that \(\dim M/G=1\).
  Then \(G\) has a finite covering group of the form \(\prod SU(2)\) or \(\prod SU(2)\times S^1\).
  Furthermore \(M\) is diffeomorphic to a \(S^2\)-bundle over a product of two-spheres.
\end{theorem}

The proofs of these theorems are based on Hauschild's study \cite{0623.57024} of spaces of q-type.
A space of \(q\)-type is defined to be a topological space \(X\) satisfying the following cohomological properties:
\begin{itemize}
\item The cohomology ring \(H^*(X)\) is generated as a \(\Q\)-algebra by elements of degree two, i.e. \(H^*(X)=\Q[x_1,\dots,x_n]/I_0\) and \(\deg x_i=2\).
\item The defining ideal \(I_0\) contains a definite quadratic form \(Q\).
\end{itemize}

The note is organised as follows.
In section~\ref{sec:qt_vanish} we show that a quasitoric manifold \(M\) with \(p_1(M)=-\sum a_i^2\) for some \(a_i\in H^2(M)\) is of \(q\)-type.
In section~\ref{sec:qt_hom} we prove Theorem \ref{thm1}.
In section~\ref{sec:cohom1} we recall some properties of cohomogeneity one manifolds.
In section~\ref{sec:quasi} we prove Theorem \ref{sec:thm2}.

The results presented in this note form part of the outcome of my Ph.D. thesis \cite{wiemeler_phd} written under the supervision of Prof. Anand Dessai at the University of Fribourg.
I would like to thank Anand Dessai for helpful discussions.

\section{Quasitoric manifolds with $p_1(M)=-\sum a_i^2$}
\label{sec:qt_vanish}
In this section we study quasitoric manifolds \(M\) with \(p_1(M)=-\sum a_i^2\) for some \(a_i\in H^2(M)\).
To do so we first introduce some notations from \cite{0623.57024} and \cite[Chapter VII]{0429.57011}.
For a topological space \(X\) we define the topological degree of symmetry of \(X\) as
\begin{equation*}
  N_t(X)=\max \{\dim G; G\text{ compact Lie-group, } G\text{ acts effectively on }X\}
\end{equation*}
Similarly one defines the semi-simple degree of symmetry of \(X\) as
\begin{equation*}
  N^{ss}_t(X)=\max \{\dim G; G\text{ compact semi-simple Lie-group, } G\text{ acts effectively on }X\}
\end{equation*}
and the torus-degree of symmetry as
\begin{equation*}
  T_t(X)=\max \{\dim T; T\text{ torus, } T\text{ acts effectively on }X\}.
\end{equation*}
In the above definitions we assume that all groups act continuously.

Another important invariant of a topological space \(X\) used in \cite{0623.57024} is the so called embedding dimension of its rational cohomology ring.
For a local \(\Q\)-algebra \(A\), we denote by \(\edim A\) the embedding dimension of \(A\).
By definition, we have \(\edim A=\dim_\Q \mathfrak{m}_A/\mathfrak{m}_A^2\), where \(\mathfrak{m}_A\) is the maximal ideal of \(A\).
In case that \(A=\bigoplus_{i\geq 0}A^i\) is a positively graded local \(\Q\)-algebra, \(\mathfrak{m}_A\) is the augmentation ideal \(A_+=\bigoplus_{i>0}A^i\).
If furthermore \(A\) is generated by its degree two part, then \(\mathfrak{m}_A^2=\bigoplus_{i>2}A^i\).
Therefore for a quasitoric manifold \(M\) over the polytope \(P\) we have \(\edim H^*(M)=\dim_\Q H^2(M)=m-n\) where \(m\) is the number of facets of \(P\) and \(n\) is its dimension.

\begin{lemma}
\label{sec:quas-manif-with-2}
  Let \(M\) be a quasitoric manifold with \(p_1(M)=-\sum a_i^2\) for some \(a_i\in H^2(M)\).
  Then \(M\) is a manifold of q-type.
\end{lemma}
\begin{proof}
  The discussion at the beginning of section 3 of \cite{semifree}  together with Corollary 6.8 of \cite[p. 448]{davis91:_convex_coxet} shows that there are
  a basis \(u_{n+1},\dots,u_{m}\) of \(H^2(M)\) and
  \(\lambda_{i,j}\in \mathbb{Z}\) such that
  \begin{equation*}
    p_1(M)=\sum_{i=n+1}^{m}u_i^2 +\sum_{j=1}^n\left(\sum_{i=n+1}^{m}\lambda_{i,j}u_i\right)^2.
  \end{equation*}
  Therefore
  \begin{align*}
  	0&= \sum_{i=n+1}^{m}u_i^2 +\sum_{j=1}^n\left(\sum_{i=n+1}^{m}\lambda_{i,j}u_i\right)^2 +\sum_i a_i^2\\
  	&=\sum_{i=n+1}^{m}u_i^2 +\sum_{j=1}^n\left(\sum_{i=n+1}^{m}\lambda_{i,j}u_i\right)^2 + \sum_j\left(\sum_{i=n+1}^{m}\mu_{i,j}u_i\right)^2 
  \end{align*}
 with some \(\mu_{i,j} \in\Q\) follows.
  
  Because 
  \begin{equation*}
    \sum_{i=n+1}^{m}X_i^2 +\sum_{j=1}^n\left(\sum_{i=n+1}^{m}\lambda_{i,j}X_i\right)^2+\sum_j\left(\sum_{i=n+1}^{m}\mu_{i,j}X_i\right)^2 
  \end{equation*}
  is a positive definite bilinear form the statement follows.
\end{proof}

\begin{prop}
\label{sec:gg}
  Let \(M\) be a quasitoric manifold of q-type over the \(n\)-dimensional polytope \(P\).
  Then we have for the number \(m\) of facets of \(P\):
  \begin{equation*}
    m\geq 2n
  \end{equation*}
\end{prop}
\begin{proof}
  By Theorem 3.2 of \cite[p. 563]{0623.57024}, we have
  \begin{equation*}
    n \leq T_t(M) \leq \edim H^*(M) =m-n.
  \end{equation*}
  Therefore we have \(2n\leq m\).
\end{proof}

\begin{remark}
  The inequality in the above proposition is sharp, because for \(M=S^2\times\dots\times S^2\) we have \(m=2n\) and \(p_1(M)=0\).
\end{remark}

By Theorem 5.13 of \cite[p. 573]{0623.57024}, we have for a manifold \(M\) of q-type that \(N_t^{ss}\leq \dim M + \edim M\). Hence, for a quasitoric manifold \(M\), we get: 
\begin{prop}
\label{sec:quasitoric-manif}
  Let \(M\) as in Proposition~\ref{sec:gg}. Then we have
  \begin{equation*}
    N_t^{ss}(M)\leq 2n+m-n=n+m.
  \end{equation*}
\end{prop}

\begin{remark}
  The inequality in the above proposition is sharp because for \(M= S^2\times \dots\times S^2\) we have \(m=2n\) and \(SU(2)\times\dots\times SU(2)\) acts on \(M\) and has dimension \(3n\).
\end{remark}

\section{Quasitoric manifolds which are also homogeneous spaces}
\label{sec:qt_hom}

In this section we prove Theorem~\ref{thm1}.
Recall from Lemma~\ref{sec:quas-manif-with-2} that a quasitoric manifold \(M\) with  first Pontrjagin-class equal to the negative of the sum of squares of elements of \(H^2(M)\) is a manifold of q-type.

Let \(M\) be a quasitoric manifold over the polytope \(P\) which is also a homogeneous space and is of q-type.

Let \(G\) be a compact connected Lie-group and \(H\subset G\) a closed subgroup such that \(M\) is homeomorphic or diffeomorphic to \(G/H\).
Because \(\chi(M)>0\) and \(M\) is simply connected, we have \(\rank G=\rank H\) and \(H\) is connected.
Therefore we may assume that \(G\) is semi-simple and simply connected.

Let \(T\) be a maximal torus of \(G\).
Then \(\left(G/H\right)^T\) is non-empty.
By Theorem 5.9 of \cite[p. 572]{0623.57024}, the isotropy group \(G_x\) of a point \(x\in\left(G/H\right)^T\) is a maximal torus of \(G\). Hence, \(H\) is a maximal torus of \(G\).

Now it follows from Theorem 3.3 of \cite[p. 563]{0623.57024} that
\begin{equation*}
  T_t(G/H) =\rank G.
\end{equation*}

Because \(M\) is quasitoric, we have \(n\leq T_t(G/H)\).
Combining these inequations, we get
\begin{equation*}
  \dim G - \dim H = \dim M = 2n \leq 2 \rank G.
\end{equation*}
This equation implies that \(\dim G\leq 3\rank G\).

For a simple simply connected Lie-group \(G'\) we have \(\dim G'\geq 3 \rank G'\) and \(\dim G'=3\rank G'\) if and only if \(G'=SU(2)\).
Therefore we have \(G=\prod SU(2)\) and \(M= \prod SU(2)/T^1=\prod S^2\).
This proves Theorem~\ref{thm1}.

\section{Cohomogeneity one manifolds}
\label{sec:cohom1}
Here we discuss some facts about closed cohomogeneity one Riemannian \(G\)-manifolds \(M\) with orbit space a compact interval \([-1,1]\).
We follow \cite[p. 39-44]{1145.53023} in this discussion.

We fix a normal geodesic \(c:[-1,1]\rightarrow M\) perpendicular to all orbits.
We denote by \(H\) the principal isotropy group \(G_{c(0)}\), which is equal to the isotropy group \(G_{c(t)}\) for \(t\in ]-1,1[\), and by \(K^{\pm}\) the isotropy groups of \(c(\pm 1)\).

Then \(M\) is the union of tabular neighbourhoods of the non-principal orbits \(Gc(\pm 1)\) glued along their boundary, i.e., by the slice theorem we have
\begin{equation}
\label{eq:7}
  M= G\times_{K^-}D_-\cup G\times_{K^+}D_+,
\end{equation}
where \(D_{\pm}\) are discs.
  Furthermore \(K^{\pm}/H=\partial D_{\pm}=S_{\pm}\) are spheres.

Note that \(M\) may be reconstructed from the following diagram of groups.
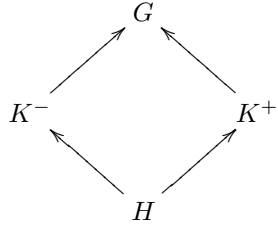
\begin{equation*}
  \xymatrix{
    &G&\\
    K^{-} \ar[ur]&& K^{+}\ar[ul]\\
    &H \ar[ur] \ar[ul]&\\
  }
\end{equation*}
The construction of such a group diagram from a cohologeneity one manifold may be reversed.
Namely, if such a group diagram with \(K^{\pm}/H=S_{\pm}\) spheres is given, then one may construct a cohomogeneity one \(G\)-manifold from it.
We also write these diagrams as \(H\subset K^{-},K^+\subset G\).

Now we give a criterion for two group diagrams yielding up to \(G\)-equivariant diffeomorphism the same manifold \(M\).

\begin{lemma}[{\cite[p. 44]{1145.53023}}]
\label{sec:cohom-one-manif}
  The group diagrams \(H\subset K^{-},K_1^+\subset G\) and  \(H\subset K^{-},K_2^+\subset G\) yield the same cohomogeneity one manifold up to equivariant diffeomorphism if there is an \(a\in N_G(H)^0\) with \(K_1^+=aK_2^+a^{-1}\).
\end{lemma}

\section{Quasitoric manifolds with cohomogeneity one actions}
\label{sec:quasi}
In this section we study quasitoric manifolds \(M\) which admit a smooth action of a compact connected Lie-group \(G\) which has an orbit of codimension one.
As before we do not assume that the \(G\)-action on \(M\) extends the torus action.
We have the following lemma:

\begin{lemma}
\label{sec:quas-manif-with}
  Let \(M\) be a quasitoric manifold of dimension \(2n\) which is of q-type.
Assume that the compact connected Lie-group \(G\) acts almost effectively and smoothly on \(M\) such that \(\dim M/G =1\).
Then we have:
\begin{enumerate}
\item\label{item:kohom1} The singular orbits are given by \(G/T\) where \(T\) is a maximal torus of \(G\).
\item\label{item:kohom2} The Euler-characteristic of \(M\) is \(2\# W(G)\).
\item\label{item:kohom3} The principal orbit type is given by \(G/S\), where \(S\subset T\) is a subgroup of codimension one.
\item\label{item:kohom4} The center \(Z\) of \(G\) has dimension at most one.
\item\label{item:kohom5} \(\dim G/T = 2n-2\).
\end{enumerate}
\end{lemma}
\begin{proof}
  At first note that \(M/G\) is an interval \([-1,1]\) and not a circle because \(M\) is simply connected.
  We start with proving (\ref{item:kohom1}).
  Let \(T\) be a maximal torus of \(G\).
  By passing to a finite covering group of \(G\) we may assume \(G=G'\times Z'\) with \(G'\) a compact connected semi-simple Lie-group and \(Z'\) a torus.
  Let \(x\in M^T\).
  Then the isotropy group \(G_x\) has maximal rank in \(G\).
  Therefore \(G_x\) splits as \(G_x'\times Z'\).
  
  By Theorem 5.9 of \cite[p. 572]{0623.57024}, \(G_x'\) is a maximal torus of \(G'\).
  Therefore we have \(G_x=T\).

  Because \(\dim G - \dim T\) is even, \(x\) is contained in a singular orbit.
  In particular we have
  \begin{equation}
    \label{eq:hom1}
    \chi(M)=\chi(M^T)=\chi(G/K^+)+\chi(G/K^-),
  \end{equation}
  where \(G/K^\pm\) are the singular orbits.
  Furthermore we may assume that \(G/K^+\) contains a \(T\)-fixed point.
  This implies
  \begin{equation}
    \label{eq:hom2}
    \chi(G/K^+)=\chi(G/T)=\# W(G) =\# W(G').
  \end{equation}

  Now assume that all \(T\)-fixed points are contained in the singular orbit \(G/K^+\).
  Then we have \((G/K^-)^T=\emptyset\).
  This implies
  \begin{equation*}
    \chi(M)=\chi(G/K^+)=\# W(G').
  \end{equation*}
  Now Theorem 5.11 of \cite[p. 573]{0623.57024} implies that \(M\) is the homogeneous space \(G'/G'\cap T =G/T\).
  This contradicts our assumption that \(\dim M/G=1\).

  Therefore both singular orbits contain \(T\)-fixed points.
  This implies that they are of type \(G/T\).
  This proves (\ref{item:kohom1}).
  (\ref{item:kohom2}) follows from (\ref{eq:hom1}) and (\ref{eq:hom2}).

  Now we prove (\ref{item:kohom3}) and (\ref{item:kohom5}).
  Let \(S\subset T\) be a minimal isotropy group.
  Then \(T/S\) is a sphere of dimension \(\codim (G/T,M)-1\).
  Therefore \(S\) is a subgroup of codimension one in \(T\) and \(\codim (G/T,M)=2\).

  If the center of \(G\) has dimension greater than one, then \(\dim Z'\cap S \geq 1\).
  That means that the action is not almost effective. 
  Therefore (\ref{item:kohom4}) holds.
\end{proof}

By Lemma~\ref{sec:quas-manif-with}, we have with the notation of the previous section that \(K^{\pm}\) are maximal tori of \(G\) containing \(H=S\).
In the following we will write \(G=G'\times Z'\) with \(G'\) a compact connected semi-simple Lie-group and \(Z'\) a torus.

Because  \(K^{\pm}\) are maximal tori of the identity component \(Z_G(S)^0\) of the centraliser of \(S\), there is some \(a\in Z_G(S)^0\) such that \(K^-=aK^+a^{-1}\).
By Lemma~\ref{sec:cohom-one-manif}, we may assume that \(K^+=K^-=T\).
Now from Theorem 4.1 of \cite[p. 198]{0599.57016} it follows that \(M\) is
a fiber bundle over \(G/T\) with fiber the cohomogeneity one manifold with group diagram \(S\subset T,T\subset T\).
Therefore it is a \(S^2\)-bundle over \(G/T\).

\begin{lemma}
\label{sec:quas-manif-with-1}
  Let \(M\) and \(G\) as in the previous lemma. Then we have
  \begin{equation*}
    T_t(M)\leq \rank G' +1.
  \end{equation*}
\end{lemma}
\begin{proof}
  At first we recall the rational cohomology of \(G/T\).
  By \cite[p. 67]{0158.20503}, we have
  \begin{equation*}
    H^*(G/T)\cong H^*(BT)/I
  \end{equation*}
  where \(I\) is the ideal generated by the elements of positive degree which are invariant under the action of the Weyl-group of \(G\).
  Therefore it follows that
  \begin{align*}
    \dim_\Q H^\text{odd}(G/T)&=0& \text{and}& & \dim_\Q H^2(G/T)&=\rank G'.
  \end{align*}
Therefore the Serre spectral sequence for the fibration \(S^2 \rightarrow M \rightarrow G/T\) degenerates.
Hence, we have
\begin{equation*}
H^*(M)=H^*(G/T)\otimes H^*(S^2)
\end{equation*}
as \(H^*(G/T)\)-modules.
In particular, we have
\begin{equation*}
\dim_\Q H^2(M)= \dim_\Q H^2(G/T) + \dim_\Q H^2(S^2) = \rank G' + 1.
\end{equation*}
Therefore
\begin{equation*}
T_t(M)\leq \edim H^*(M) = \dim_\Q H^2(M)= \rank G' + 1
\end{equation*}
follows.
\end{proof}

\begin{theorem}
\label{sec:quas-manif-with-3}
  Let \(M\) and \(G\) as in the previous  lemmas.
  Then \(G\) has a finite covering group of the form \(\prod SU(2)\) or \(\prod SU(2)\times S^1\).
  Furthermore \(M\) is diffeomorphic to a \(S^2\)-bundle over a product of two-spheres.
\end{theorem}
\begin{proof}
  Because \(M\) is quasitoric we have \(n\leq T_t(M)\).
  By Lemma~\ref{sec:quas-manif-with} we have
  \begin{equation*}
    \dim G'-\rank G'=\dim G/T =2n-2.
  \end{equation*}
  Now Lemma~\ref{sec:quas-manif-with-1} implies
  \begin{equation*}
    \dim G'= 2n-2 + \rank G'\leq 3 \rank G'.
  \end{equation*}
Therefore \(\prod SU(2)\) is a finite covering group of \(G'\).
This implies the statement about the finite covering group of \(G\).

It follows that \(G/T=\prod S^2\).
Therefore \(M\) is a \(S^2\)-bundle over \(\prod S^2\).
\end{proof}

Now Theorem~\ref{sec:thm2} follows from Theorem~\ref{sec:quas-manif-with-3} and Lemma~\ref{sec:quas-manif-with-2}.

\bibliography{diss}{}

\providecommand{\bysame}{\leavevmode\hbox to3em{\hrulefill}\thinspace}
\providecommand{\MR}{\relax\ifhmode\unskip\space\fi MR }
\providecommand{\MRhref}[2]{%
  \href{http://www.ams.org/mathscinet-getitem?mr=#1}{#2}
}
\providecommand{\href}[2]{#2}
\begin{thebibliography}{10}

\bibitem{0158.20503}
A.~Borel, \emph{{Topics in the homology theory of fibre bundles.}},
  {Berlin-Heidelberg-New York: Springer-Verlag}, 1967 (English).

\bibitem{davis91:_convex_coxet}
M.~Davis and T.~Januszkiewicz, \emph{Convex polytopes, coxeter orbifolds and
  torus actions}, Duke Math. J. \textbf{62} (1991), no.~2, 417--451.

\bibitem{1145.53023}
K.~Grove, B.~Wilking, and W.~Ziller, \emph{{Positively curved cohomogeneity one
  manifolds and 3-Sasakian geometry.}}, J. Differ. Geom. \textbf{78} (2008),
  no.~1, 33--111 (English).

\bibitem{0623.57024}
V.~Hauschild, \emph{{The Euler characteristic as an obstruction to compact Lie
  group actions.}}, Trans. Am. Math. Soc. \textbf{298} (1986), 549--578
  (English).

\bibitem{0429.57011}
W.~Y. Hsiang, \emph{{Cohomology theory of topological transformation groups.}},
  {Ergebnisse der Mathematik und ihrer Grenzgebiete. Band 85.
  Berlin-Heidelberg-New York: Springer-Verlag.}, 1975 (English).

\bibitem{kuroki_pre_2_2009}
S.~Kuroki, \emph{Classification of quasitoric manifolds with codimension one
  extended actions}, Preprint (2009).

\bibitem{kuroki_pre_1_2009}
\bysame, \emph{Characterization of homogeneous torus manifolds}, Osaka J. Math.
  \textbf{47} (2010), no.~1, 285--299 (English).

\bibitem{semifree}
M.~Masuda and T.~E. Panov, \emph{Semi-free circle actions, {B}ott towers, and
  quasitoric manifolds}, Mat. Sb. \textbf{199} (2008), no.~8, 95--122.

\bibitem{0599.57016}
J.~Parker, \emph{{4-dimensional G-manifolds with 3-dimensional orbits.}}, Pac.
  J. Math. \textbf{125} (1986), 187--204 (English).

\bibitem{wiemeler_phd}
M.~Wiemeler, \emph{On the classification of torus manifolds with and without
  non-abelian symmetries}, Ph.D. thesis, University of Fribourg, 2010.

\end{thebibliography}
\bibliographystyle{amsplain}
\end{document}